\documentclass[12pt]{amsart}

\usepackage{amssymb}
\usepackage{amsmath}
  \usepackage{paralist}
  \usepackage{mathrsfs}
    \usepackage{color}

 \usepackage[colorlinks=true]{hyperref}
  \hypersetup{urlcolor=blue, citecolor=red}

\numberwithin{equation}{section}

\newtheorem{theorem}{Theorem}
\newtheorem{lemma}{Lemma}

\newtheorem{remark}{Remark}

\numberwithin{theorem}{section}
\numberwithin{corollary}{section}
\numberwithin{lemma}{section}
\numberwithin{definition}{section}
\numberwithin{proposition}{section}
\numberwithin{remark}{section}

\def\u{\bmit u}
\def\a{\bmit a}
\def\b{\bmit b}
\def\w{\bmit w}
\def\x{ \bmit x}
\def\e{\bmit e}
\def\f{\bmit f}
\def\n{\bmit n}

\def\s{\bmit s}
\def\w{\bmit w}

\def\h{\bmit h}
\def\vi{\bmit v}

\def\L{ \bmit L}

\def\Q{\bmit G}

\def\0{\bf 0}
\def\into{\int\limits_}

\def\div{{\hbox{\rm div}\,}}

\def\Ci{\textsf{\textbf{C}}}
\def\C{\textsf{\textbf{C}}}
\def\Cc{\textsf{\text{C}}}
 
\def\bmit{\boldsymbol}

\def\XXint#1#2#3{{\setbox0=\hbox{$#1{#2#3}{\int}$ }
\vcenter{\hbox{$#2#3$ }}\kern-.6\wd0}}

\begin{document}

\title{The Stokes paradox in inhomogeneous elastostatics}

\author{Adele Ferone}
\address{Dipartimento di Matematica e Fisica, Universit\`a degli Studi della Campania ``L. Vanvitelli'', Viale Lincoln 5, 81100 Caserta, Italy}
\curraddr{}
\email{adele.ferone@unicampania.it}
\thanks{}

\author{Remigio Russo}
\address{Dipartimento di Matematica e Fisica, Universit\`a degli Studi della Campania ``L. Vanvitelli'', Viale Lincoln 5, 81100 Caserta, Italy}
\curraddr{}
\email{remigio.russo@unicampania.it}
\thanks{}

\author{Alfonsina Tartaglione}
\address{Dipartimento di Matematica e Fisica, Universit\`a degli Studi della Campania ``L. Vanvitelli'', Viale Lincoln 5, 81100 Caserta, Italy}
\curraddr{}
\email{alfonsina.tartaglione@unicampania.it}
\thanks{}



\date{}

\dedicatory{}

\begin{abstract}
 
We prove that the displacement problem  of  
inhomogeneous elastostatics in a two--dimensional exterior Lipschitz domain   has a unique solution with finite Dirichlet integral $\u$, 
vanishing uniformly at infinity if and only if the boundary  datum satisfies a suitable  compatibility condition (Stokes' paradox).  Moreover,
we prove that it is unique under the sharp condition $\u=o(\log r)$   and decays uniformly at
infinity with a rate depending on the   elasticities. In particular, if  these last ones tend to a homogeneous state at large distance, then $\u=O(r^{-\alpha})$, for every $\alpha<1$.
\end{abstract}

\maketitle

\section{Introduction}

Let $\Omega$ be an exterior Lipschitz domain of ${\Bbb R}^2$.  The
 displacement  problem of plane elastostatics in exterior domains is   to find a solution
to the equations
\begin{equation}
\label{conell}
\begin{array}{r@{}l}
\div\C[\nabla{\u}] & {} ={\0}\quad\hbox{\rm in }\Omega, \\[2pt]
{\u} & {} =\hat{\u} \quad\hbox{\rm on }\partial\Omega, \\[2pt]
\displaystyle\lim_{r\to+\infty}{\u}(x) & {} ={ \0}, 
\end{array}  
\end{equation}
 where ${\u}$ is the (unknown)  displacement field, $\hat{\u}$ is an (assigned)  boundary displacement, ${\C}\equiv [\Cc_{ijhk}]$ is the (assigned) elasticity
tensor, $i.e.$,   a map  from
${\Omega}\times\hbox{\rm Lin}\to\hbox{\rm Sym}$, linear on  Sym and vanishing in $\Omega\times{\rm Skw}$. We shall  assume $\C$ to be symmetric, $i.e.$, $\Cc_{ijhk}=\Cc_{hkij}$ and positive definite, $i.e.$,
\begin{equation}
\label{Sxxa}
\mu_0|{\bmit E}|^2\le{\bmit E}\cdot\C[{\bmit E}]\le \mu_e|{\bmit E}|^2,\quad \forall\,{\bmit E}\in{\rm Sym}, \;\;\hbox{\it a.e. {\rm in}}\  \Omega.
\end{equation}
By appealing to the principle of virtual work and taking into account that
${\bmit\varphi}\in C^\infty_{ 0}(\Omega)$ is an
admissible (or virtual) displacement, we say that ${\u}\in W^{1,q}_{ \rm loc}(\Omega )$
is a weak solution (variational solution  for $q=2$)   to (\ref{conell})$_1$  provided
\begin{equation*}
\into \Omega\nabla{\bmit\varphi}\cdot{\C}[\nabla{\u}]=0, \quad
\forall{\bmit\varphi}\in C^\infty_{ 0}(\Omega). 
\end{equation*} 
  A weak solution to  (\ref{conell})  is a weak
solution to (\ref{conell})$_1$ which satisfies the boundary condition in the sense of the trace in
Sobolev's spaces and tends to zero at infinity in a   generalized sense. 
If ${\u}\in W^{1,q}_{\rm loc}(\overline\Omega )$ is a weak solution to (\ref{conell})  the traction field on the boundary
$$
\s({\u})={\C}[\nabla{\u}]\n 
$$ 
exists as a well defined field of
$W^{-1/q,q}(\partial\Omega)$ and for $q=2$  the following generalized   {\it work and energy
relation\/} \cite{Gurtin} holds
\begin{equation*}
\into {\Omega_R}\nabla{\u}\cdot{\C}[\nabla{\u}]=\into {\partial\Omega}\u\cdot\s(\u) 
+\into {\partial S_R}\u\cdot\s(\u),
\end{equation*}
for every  large $R,$ where  with abuse of notation  by $\int_{\Sigma} \u\cdot\s(\u)$ we mean the value of the
functional
$\s(\u)\in
W^{-1/2,2}(\Sigma)$ at $\u\in W^{1/2,2}(\Sigma)$ and $\n$ is the unit outward  (with
respect to $\Omega$) normal  to $\partial\Omega$.   It will be clear from the context when we shall refer to an ordinary integral or to a functional.

It is a routine to show  that under assumption (\ref{Sxxa}), $(\ref{conell})_{1,2}$ has a unique solution
$\u\in D^{1,2}(\Omega)$, we shall call {\it D--solution\/} (for the notation see at the end of this section). Moreover, it exhibits   more regularity provided ${\C}$, $\partial\Omega$ and $\hat{\u}$ are more
regular. In particular, the following well--known theorem holds \cite{Giusti}, \cite{KO}.
\begin{theorem}
\label{Extheon}
Let $\Omega$ be an exterior Lipschitz domain of ${\Bbb R}^2$ and let ${\C}$ satisfy  $(\ref{Sxxa})$\footnote{ For constant $\C$ (homogeneous elasticity) it is sufficient to assume that $\Ci$ is strongly elliptic, $i.e.$, there is $\lambda_0>0$ such that $\lambda_0|\a|^2|\b|^2\le \a\cdot\Ci[\a\otimes\b]\b$, for all $\a,\b\in{\Bbb R}^2$ .}. If
$\hat{\u}\in W^{1/2,2}(\partial\Omega)$, then $(\ref{conell})_{1,2}$ has a  unique D--solution $\u$ which is locally H\"older continuous in $\Omega$.
Moreover, if
$\Omega$ is of class $C^k$, ${\C}\in C^{k-1}_{\rm loc}(\overline\Omega )$  and $\hat{\u}\in
W^{k-1/q,q}(\partial\Omega)$ $(k\ge 1, q\in(1,+\infty))$, then $\u\in W^{k,q}_{\rm loc}(\overline\Omega )$.
\end{theorem}
 
The main problem left open by Theorem \ref{Extheon} is to establish the behavior of the variational
solution at large distance: {\it does  $\u$ converge  to  a constant vector at infinity and,
if so, does  $($or under what conditions and in what sense$)$  $\u$ satisfies
$(\ref{conell})_{3}$?} For constant  $\C$ (homogeneous elasticity) the  situation is well understood (see, $e.g.$, \cite{RussoPG}, \cite{RSima}), at least in its negative information. Indeed, a solution to  (\ref{conell})$_{1,2}$ is expressed by a simple layer potential
$$
\u(x)=\vi[{\bmit\psi}](x)+{\bmit\kappa}, 
$$
for some ${\bmit\psi}\in W^{-1/2,2}(\partial\Omega)$,
where
$$
\vi[{\bmit\psi}](x) =\into{\partial\Omega}\boldsymbol{\mathscr{U}}(x-y){\bmit\psi} (y)ds_y
$$
 is the simple layer  with density ${\bmit\psi}$ such that
 \begin{equation}
\label{DAdob}
\into{\partial\Omega}{\bmit\psi}={\bf 0}
\end{equation}
and
$$
\boldsymbol{\mathscr{U}}(x-y)= {\boldsymbol{\bmit\Phi}}_0\log|x-y|+{\boldsymbol{\bmit\Phi}}(x-y),
$$
with $ {\boldsymbol{\bmit\Phi}}_0\in {\rm Lin}$ and ${\boldsymbol{\bmit\Phi}}:{\Bbb R}^2\setminus\{o\}\rightarrow{\rm Lin}$ homogeneous of degree zero, is the fundamental solution to equations (\ref{conell}) (see, $e.g.$, \cite{John}). The space $\boldsymbol{\frak C}=\{{\bmit\psi}\in L^2(\partial\Omega):{\bmit v}[{\bmit\psi}]_{\vert\partial\Omega}=\hbox{\rm constant}\}$ has dimension two and if $\{{\bmit\psi}_1,{\bmit\psi}_2\}$ is a basis of $\boldsymbol{\frak C}$, then $\{\int_{\partial\Omega}{\bmit\psi}_1,\int_{\partial\Omega}{\bmit\psi}_2\}$ is a basis of ${\Bbb R}^2$; (\ref{DAdob}) assures that $\u-{\bmit\kappa}=O(r^{-1})$, where the constant vector ${\bmit\kappa}$ is determined by the relation
\begin{equation*}
\into{\partial\Omega}(\hat{\u}-{\bmit\kappa})\cdot {\bmit\psi}'=0\quad\forall{\bmit\psi}'\in\boldsymbol{\frak C}.
\end{equation*}
Hence it follows
\begin{theorem}
\label{Exxxxon}
Let $\Omega$ be an exterior Lipschitz domain of $\,{\Bbb R}^2$ and let ${\C}$ be  
constant and strongly elliptic. If
$\hat{\u}\in W^{1/2,2}(\partial\Omega)$, then $(\ref{conell})$ 
has a unique  D--solution, analytic in $\Omega$, 
 if and only if 
 \begin{equation}
\label{DCCCob}
\into{\partial\Omega}\hat{\u} \cdot {\bmit\psi}'=0,\quad\forall{\bmit\psi}'\in\boldsymbol{\frak C}\;\Leftrightarrow\;\into{\partial\Omega}\s(\u)={\bf 0}.
\end{equation}
Moreover, $\u$ is  unique in the class 
  \begin{equation}
\label{DCFdb}
 \{\u\in W^{1,2}_{\rm loc}(\overline\Omega): \u=o(\log r)\}
\end{equation}
and modulo a field   ${\bmit v}[{\bmit\psi}']- {\bmit v}[{\bmit\psi}']_{\vert\partial\Omega}$, ${\bmit\psi}'\in\boldsymbol{\frak C}$, in the class 
  \begin{equation*}
 \{\u\in W^{1,2}_{\rm loc}(\overline\Omega):\u=o(r)\}.
\end{equation*}

\end{theorem}

An immediate consequence of (\ref{DCCCob})  is nonexistence of a solution  to (\ref{conell})   corresponding to a constant boundary data. This phenomenon for the Stokes' equations
$$
\begin{array}{r@{}l}
\mu\Delta\u-\nabla p& {} ={\0} , \\[2pt]
\div{\u} & {} =0,
\end{array}
$$
is popular as {\it Stokes paradox\/} and goes back to the pioneering work of G.G. Stokes (1851) on the study of the (slow) translational motions of a ball in an incompressible viscous fluid of viscosity $\mu$ (see \cite{GS} and Ch. V of \cite{Galdi}). Clearly, as it stands, Stokes' paradox can be read only as a {\it negative result\/}, unless we are not able to find an analytic expression of  the densities of $\boldsymbol{\frak C}$. As far as we know, this is possible only for the ellipse of equation $f(\xi)=1$. Indeed, in this case it is known that $\boldsymbol{\frak C}=\hbox{\rm spn}\,\{ \e_1/|\nabla f|,\e_2/|\nabla f|\}$ (see, $e.g.$, \cite{Vassiliev})  and Theorem \ref{Exxxxon} reads
\begin{theorem}
\label{Esern}
Let $\Omega$ be the exterior of an ellipse  of equation $f(\xi)=1$ and let ${\C}$ be  
constant and strongly elliptic.  If
$\hat{\u}\in W^{1/2,2}(\partial\Omega)$, then $(\ref{conell})$ 
has a unique solution expressed by a simple layer potential, with a density ${\bmit\psi}\in W^{-1/2,2}(\partial\Omega)$ satisfying $(\ref{DAdob})$,
 if and only if 
 \begin{equation*}
\into{\partial\Omega}  {\hat{\u}\over|\nabla f|}={\bf 0}. 
\end{equation*}
\end{theorem}

The situation is not so clear in inhomogeneous elasticity. In fact, in such a case  it is
not known whether
$\u$ converges at infinity and  even the definition of  the space $\boldsymbol{\frak C} $  needs to be clearified.  

The purpose of this paper is to show that results similar to those stated in Theorem \ref{Exxxxon}  hold in inhomogeneous elasticity, at least in its negative meaning.

By $\boldsymbol{\frak M}$ we shall denote the   linear space  of variational solutions to
\begin{equation*}
\begin{array}{r@{}l}
\div{\C}[\nabla \h]& {} ={\0}\!\qquad\quad\hbox{\rm in }\Omega, \\[2pt]
  \h &{} ={\0}\!\qquad\quad\hbox{\rm on }\partial\Omega,
\\[2pt]
\h &{}\in BMO. 
\end{array}  
\end{equation*}

We say that \hbox{\it $\Ci$ is regular at infinity\/} if there is a constant elasticity tensor $\Ci_0$ such that 
\begin{equation}
\label{lkj˜jkskj9}
\lim_{|x|\to+\infty} {\C}(x) ={\C}_0.
\end{equation}

\medskip

The following theorem holds.

\begin{theorem} 
\label{T1}{\rm (Stokes' Paradox of inhomogeneous elastostatics)} -- 
Let $\Omega$ be an exterior Lipschitz  domain of ${\Bbb R}^2$ and  let  
${\C}$ satisfy $(\ref{Sxxa})$. It holds:

\medskip

$(i)$  $\dim\boldsymbol{\frak M}=2$ and if  $\{\h_1,\h_2\}$ is a basis of $\boldsymbol{\frak M}$, then $\Big\{\into {\partial\Omega}\s(\h_1), \into {\partial\Omega}\s(\h_2)\Big\}$ is a basis of ${\Bbb R}^2$.

\medskip

$(ii)$  If 
$\hat{\u}\in W^{1/2,2}(\partial\Omega)$, then system
$(\ref{conell}) $ has a unique D--solution $\u$ if and only if 
\begin{equation}
\label{defu0}
\into {\partial\Omega}\hat{\u}\cdot\s(\h)={0},\quad\forall\, \h\in\boldsymbol{\frak M}.
\end{equation}

\medskip

$(iii)$ $\u$ is unique in the class $(\ref{DCFdb})$ and modulo a field  $\h\in\boldsymbol{\frak M}$ in the class 
  \begin{equation*}
 \{\u\in W^{1,2}_{\rm loc}(\overline\Omega): \u=o(r^{\gamma/2})\},
\end{equation*}
where
 \begin{equation*}
\gamma={4\mu_0\over 5\mu_0+8\mu_e}
\end{equation*}

\medskip

$(iv)$  there is a positive $\alpha$ depending on the elasticities such that
\begin{equation}
\label{desss1}
\u=O(r^{-\alpha})
\end{equation}
Moreover, if  $\Ci$ is regular at infinity
 then $(\ref{desss1})$ holds for all $\alpha<1$.
 
\end{theorem}

Clearly,  $(i)-(ii)$   imply in particular that if $\hat{\u}$ is constant, then (\ref{conell})
has no solution in $D^{1,2}(\Omega)$ (Stokes' paradox).

Also,  for more particular tensor $\C$ we prove 
\begin{theorem} 
\label{T3}
Let $\Omega$ be an exterior Lipschitz  domain of $\,{\Bbb R}^2$ and let ${\C}:\Omega\times{\rm Lin}\to {\rm Lin}$ satisfies 
\begin{equation}
\label{SEaabb}
 \lambda |{\bmit E}|^2\le {\bmit E}\cdot\C[{\bmit E}]\le \Lambda|{\bmit E}|^2,\quad \forall\,{\bmit E} \in\hbox{\rm Lin}.
\end{equation}
A variational solution  to the system
\begin{equation*}
\begin{array}{r@{}l}
\div\C[\nabla{\u}] & {} ={\0}\quad\hbox{\rm in }\Omega, \\[2pt]
{\u} & {} =\hat{\u} \quad\hbox{\rm on }\partial\Omega,  \end{array}  
\end{equation*}
 is unique in the class
 \begin{equation}
 \label{CCCEX}
\{\u: \u=o(r^{1/\sqrt L})\}\setminus{\frak M},\quad L=\Lambda/\lambda,
 \end{equation} 
and if $\u$ belongs to $D^{1,2}(\complement S_{R_0})$, then  
$$
\u-\u_0=O(r^{\epsilon-1/\sqrt L}),
$$ 
for all positive $\epsilon$, where $\u_0$ is the constant vector defined by
\begin{equation*}
\into{\partial\Omega}(\hat{\u}-\u_0)\cdot\s(\h)=0,\quad\forall\h\in{\frak M}.
 \end{equation*} 
\end{theorem}

Theorem \ref{T1}, \ref{T3} are  proved in section \ref{PT4}.  In section \ref{kkklo} we collect the
main  tools we shall need to prove them and in Section \ref{DeGC} by means of   counter--examples we observe that  our results are sharp; for instance, in Theorem \ref{T3} uniqueness fails in the class defined by (\ref{CCCEX}) with $O$ instead  of $o$.
 
 \medskip
 {\sc Notation} --  Unless otherwise specified, we will essentially use  the notation of the classical
monograph \cite{Gurtin} of M.E. Gurtin. In indicial notation $(\div{\C}[\nabla \u])_i = \partial_j(\Cc_{ijhk}\partial_k u_h)$. Lin is  the space of second--order tensors
 (linear maps from ${\Bbb R}^2$ into itself) and Sym, Skw are the spaces of the symmetric and skew elements of Lin respectively. As is customary, if ${\bmit E}\in{\rm Lin}$ and ${\bmit v}\in{\Bbb R}^2$, ${\bmit E}{\bmit v}$ is the vector with components $E_{ij}v_j$ and $\hat\nabla\u$, $\tilde\nabla\u$ denote respectively the symmetric and skew parts of $\nabla\u$.
$(o,(\e_i\}_{i=1,2})$, $o\in\Omega'$, is the standard orthonormal reference frame; $\x=x-o$, $r=|\x|$, 
$S_R =\{x\in{\Bbb R}^2:r<R\}$, $T_R =S_{2R} \setminus S_R $;
$\Omega_R=\Omega\cap S_R$; 
$R_0$ is a large positive constant such that
$S_{R_0}\supset\overline{\Omega'}$; $\e_r=\x/r$, for all $x\ne o$.
 $W^{k,q}(\Omega)$ ($k\in{\Bbb N}_0$, $q\in (1,+\infty)$)  denotes the ordinary Sobolev's
space \cite{Giusti};
$W^{k,q}_{\rm loc}(\Omega)$ and $W^{k,q}_{\rm loc}(\overline\Omega)$ are  the spaces of all
$\varphi\in W^{k,q}(K)$ such that $\varphi\in W^{k,q}_{\rm loc}(K)$   for every compact
 $K\subset \Omega$ and  $K\subset \overline \Omega$ respectively.
$W^{1-1/q,q}(\partial\Omega)$ is the trace space of  $D^{1,q}(\Omega)=\{\varphi\in L^1_{\rm
loc}(\Omega):\|\nabla{ \varphi}\|_{L^q(\Omega)}<+\infty\}$ $(q>1)$  and
$W^{-1/q,q}(\partial\Omega)$ is its dual space. $BMO=BMO({\Bbb R}^2)=\{\varphi\in L^1_{\rm
loc}({\Bbb R}^2): \sup_R {1\over {R^2}} \int_{S_R}|\varphi-\varphi_{S_R}|<+\infty\}$. ${\mathcal H}^1({\Bbb R}^2)$ is the Hardy space. 
 As is usual, if $f(x)$ and
$\phi(r)$ are functions defined in a neighborhoof of infinity $\complement S_{R_0}$, then
$f(x)=o(\phi(r))$ and
$f(x)=O(\phi(r))$ mean respectively that $\lim_{r\to+\infty}(f/g)=0$ and    $f/g$ is bounded
in $\complement S_{R_0}$
To alleviate notation, we do not distinguish between
scalar, vector and second--order tensor space functions;  $c$ will denote a positive constant whose
numerical value is not essential to our purposes; also we let $c(\epsilon)$ denote a positive function of $\epsilon>0$ such that $\lim_{\epsilon\to0^+}c(\epsilon)=0$.

\section{Preliminary results}
\label{kkklo}
Let us collect the   main tools we shall need to prove Theorem \ref{T1} and \ref {T3}  and that have some interest in themselves.
By ${\mathcal I}$ we shall denote the exterior of a ball $S_{R_0}\Supset\complement\Omega$.
\begin{lemma}
 \label{Inequal}
 {\rm \cite{DL} \cite{KO}}
Let $\u\in D^{1,q}({\mathcal I})$, $q\in  (1,+\infty)$. If $q>2 $ then $\u/r\in L^q({\mathcal I})$ and if $q<2$, then there is a constant vector $\u_0$ such that
\begin{equation*}
\into {{\mathcal I}}{|\u-\u_0|^q\over r^q}\le c\into {\mathcal I}|\nabla\u|^q\quad\hbox{\it Hardy's inequality}.
\end{equation*}
Moreover, if $\u\in D^{1,q}({\mathcal I})$ for all $q$ in  a neighborhood of 2, then $\u=\u_0+o(1)$.
 \end{lemma}
\begin{lemma}
\label{Lem1}
If $\u$ is a variational solution to $(\ref{conell})_1$ in $S_{\bar R}$, then for all $0<\rho<R\le \bar R$,
\begin{equation}
\label{Morr}
\into{  S_\rho}|\nabla\u|^2\le c\left({\rho\over R}\right)^{\gamma}\into{  S_R}|\nabla\u|^2, \quad \gamma={4\mu_0\over 5\mu_0+8\mu_e}.
\end{equation}
\end{lemma}
\begin{proof}
Assume first that $\u$ is regular.  Taking into account that
\begin{equation}
\label{Kosss}
|\hat\nabla\u|^2-|\tilde\nabla\u|^2=\nabla\u\cdot\nabla\u^\top=\div[(\nabla\u)\u-(\div\u)\u]+|\div\u|^2,
\end{equation}
a simple computation yields
\begin{equation}
\label{ASERT}
\begin{array}{l}
\displaystyle \mu_0 G(R)=\mu_0\into{S_R}  \big(|\nabla\u|^2+|\div\u|^2\big)  \displaystyle=-\mu_0\into{\partial S_R}\e_R\cdot[\nabla\u  -(\div\u){\bf 1}] \u \\[20pt]
\;\;\displaystyle +2\mu_0\into{S_R} |\hat\nabla\u|^2\le2 \int_{\partial S_R}\u\cdot\Ci[\nabla\u]\e_R-\mu_0\into{\partial S_R}\e_R\cdot[\nabla\u  -(\div\u){\bf 1}] \u. 
\end{array}
\end{equation}
Since 
\begin{equation}
\label{Serv}
\begin{array}{l}
\displaystyle\into0^{2\pi}[(\nabla\u)^\top\e_R -(\div\u)\e_R](R,\theta)d\theta={\bf 0}, \\[15pt]
\displaystyle\into0^{2\pi}  \Ci[\nabla\u]\e_R (R,\theta)d\theta={\bf 0}, 
\end{array}
\end{equation}
by   Schwarz's inequality, Cauchy's inequality and  Wirtinger's inequality  
$$
\begin{array} {l}
\displaystyle\left| \,\into{\partial S_R }\e_R\cdot \nabla\u  \Big(\u-{1\over 2\pi}\int_0^{2\pi}\u\Big)
\right|\le  \Bigg\{\,\into{\partial S_R}\ |\nabla\u|^2\into{\partial S_R}\Big| \u-{1\over 2\pi}\int_0^{2\pi}\u\Big|^2\Bigg\}^{1/2}\\[15pt]
\;\;\;\quad\qquad\qquad\qquad\qquad\qquad\qquad\displaystyle\le R\into{\partial S_R}\ |\nabla\u|^2,\\[15pt]
\displaystyle\left| \,\into{\partial S_R } \e_R\cdot(\div\u) \Big(\u-{1\over 2\pi}\int_0^{2\pi}\u\Big)
\right|\le     R\into{\partial S_R}  |\div\u|^2+{1\over 4R}\into{\partial S_R}\Big| \u-{1\over 2\pi}\int_0^{2\pi}\u\Big|^2 ,\\[15pt]
\quad\qquad\qquad\qquad\qquad\qquad\qquad\displaystyle\le R\into{\partial S_R}|\div\u|^2+{R\over 4 } \into{\partial S_R}\ |\nabla\u|^2,\\[15pt]
\displaystyle\left| \,\into{\partial S_R }\Big(\u-{1\over 2\pi}\int_0^{2\pi}\u\Big)\cdot\Ci[\nabla\u]\e_R
\right|\le\mu_e R \into{\partial S_R}\ |\nabla\u|^2,
\end{array}
$$
and taking into account that by the basic calculus
$$
G'(R)=\into{\partial S_R}\big(|\nabla\u|^2+|\div\u|^2\big),
$$
 (\ref{ASERT}) yields
$$
\gamma G(R)\le RG'(R).
$$
Hence (\ref{Morr}) follows by a simple integration. The above argument applies to a variational  solution by a classical approximation argument (see, $e.g.$,  footnote $^{(1)}$ in \cite{PS0}).
\end{proof}

\begin{remark}  
\label{CaccioCam}
If $\u$ is a variational solution to (\ref{conell})$_1$ vanishing on $\partial\Omega$ and such that $\int_{\partial\Omega}\s(\u)={\bf 0}$, then by repeating the steps in the proof of Lemma \ref{Lem1}, it follows
\begin{equation}
\label{Morr2}
\into{\Omega_\rho}|\nabla\u|^2\le c\left({\rho\over R}\right)^{\gamma}\into{\Omega_R}|\nabla\u|^2.
\end{equation}

\end{remark}
\begin{lemma}
\label{Cacciol}
If $\u$ is a variational solution to 
\begin{equation*}
\div\C[\nabla{\u}]+{\bmit f} ={\0}\quad\hbox{\rm in }\Omega,
\end{equation*}
with ${\bmit f}$ having compact support, then for large $R$
\begin{equation}
\label{Caccioppoli}
\into{  \Omega_R}|\nabla\u|^2\le  c\Big\{{1\over R^2} \into{ T_R}| \u|^2+ \sigma(\u)\Big\},   
\end{equation}
where
$$
\sigma(\u)=2\into{\partial\Omega}\u\cdot\s(\u)-\mu_0\into{\partial\Omega}\n\cdot[\nabla\u  -(\div\u){\bf 1}] \u+2\into{\Omega}{\bmit f}\cdot\u.
$$
\end{lemma}
\begin{proof}
Let
\begin{equation}
\label{function1}
g_R(r)=\begin{cases}0,&r>2R,\\
1,& r<R,\\
 R^{-1}(2R-r),&r\in[R,2R],
\end{cases}
\end{equation}
with $S_R\supset {\rm supp} {\bmit f}$.
A standard calculation and (\ref{Kosss}) yield
\begin{equation}
\label{ASERT2}
\begin{array}{l}
\displaystyle \mu_0\into{\Omega}  g_R^2\big(|\nabla\u|^2+|\div\u|^2\big)  \displaystyle=2\mu_0\into{\Omega}g_R\nabla g_R\cdot[\nabla\u  -(\div\u){\bf 1}] \u \\[20pt]
\;\;\displaystyle +2\mu_0\into{\Omega} g_R^2|\hat\nabla\u|^2-\mu_0\into{\partial\Omega}\n\cdot[\nabla\u  -(\div\u){\bf 1}] \u 
\\[20pt]
\;\;\le2\displaystyle\into{\Omega}g_R\nabla g_R\cdot\big(\mu_0(\nabla\u  -(\div\u){\bf 1})-2\Ci[\nabla\u]\big)\u+ 2\into{\partial\Omega}\u\cdot\s(\u)
\\[20pt]
\;\;-\mu_0\displaystyle\into{\partial\Omega}\n\cdot[\nabla\u  -(\div\u){\bf 1}] \u+2\into{\Omega}{\bmit f}\cdot \u.
\end{array}
\end{equation}
By a simple   application  of Cauchy's inequality (\ref{ASERT2}) implies
$$
\into\Omega  g_R^2 |\nabla\u|^2\le c \into\Omega  |\nabla g_R|^2 |\u|^2 +\sigma(\u).
$$
Hence (\ref{Caccioppoli}) follows by  the properties of the function $g_R$.
 \end{proof}

\begin{remark}  
\label{lem3out}
Under the stronger assumption $\u$ is a $D$-solution, we can repeat the previous argument to obtain instead of \eqref{Caccioppoli} the following inequality, for $R$ sufficiently large
\begin{equation}
\label{CaccioppoliOut}
\into{\complement S_R}|\nabla\u|^2\,dx\le \frac c {R^2}\into{T_R}|\u|^2\,\,.
\end{equation}
In such case instead of the function $g_R$ we have to consider the function
\begin{equation}
\label{function2}
\eta_R(r)=\begin{cases}0,&r<R,\\
1,& r>2R,\\
 R^{-1}(r-R),&r\in[R,2R]\,,
\end{cases}
\end{equation}
 and the thesis follows similarly.
\end{remark}

\begin{lemma}
\label{Lemmm34}
Let $\u$ be a variational solution to $(\ref{conell})_1$ such that 
\begin{equation}
\label{NRT0} 
\into {\partial\Omega}\s(\u)={\0}.
\end{equation}
If
 \begin{equation}
\label{MBazzzr}
\u(x)=o(r^{\gamma/2}),
\end{equation}
then $\nabla\u\in L^2(\Omega)$ and
\begin{equation}
\label{Nxx0} 
\into{\Omega}\nabla\u\cdot\C[\nabla\u]=\into{\partial\Omega}\u\cdot\s(\u).
\end{equation}
\end{lemma}
 \begin{proof}
Let $\eta_R$ be the function defined in \eqref{function2}.
For large $\bar R$  the field
\begin{equation}
\label{Bssww}
\vi=\eta_{\bar R}\u
\end{equation}   is a variational solution to 
\begin{equation}
\label{BVswe3aw}
\div\Ci [\nabla\vi] +\f={\bf 0}\quad \hbox{\rm in }{\Bbb R}^2,
\end{equation}
with
\begin{equation}
\label{Bxvvfrw}
f_i=-\Cc_{ijhk}\partial_ku_h\partial_j\eta_{\bar R}-\partial_j(\Cc_{ijhk} u_h\partial_k\eta_{\bar R}).
\end{equation}
 Let
$\vi_1$ and $\vi_2$ be the variational  solutions to the systems
\begin{equation}
\label{Sist1}
\begin{array}{r@{}l}
 \div{\C}[\nabla \vi_1]  & {} ={ 0}\quad\hbox{\rm in
}S_R ,\\[2pt]
\vi_1& {} = \vi \quad\hbox{\rm on }\partial S_R ,
\end{array}
\end{equation}
and 
\begin{equation}
\label{Sist2}
\begin{array}{r@{}l}
 \div{\C} [\nabla \vi_2] +\f & {} ={ \0}\quad\hbox{\rm in
}S_R ,\\[2pt]
 \vi_2 & {} ={\bf 0} \quad\hbox{\rm on }\partial S_R,
\end{array}
\end{equation}
respectively, with $R>2{\bar R}$. By (\ref{Morr})
\begin{equation}
\label{Saa2}
\into{  S_\rho}|\nabla\vi_1|^2\le c\left({\rho\over R}\right)^{\gamma}\into{  S_R}|\nabla\vi_1|^2.
\end{equation}
A simple computation and the first Korn inequality
$$
\|\nabla \vi_2\|_{L^2(S_R)}\le \sqrt 2\|\hat\nabla \vi_2\|_{L^2(S_R)}
$$ yield
\begin{equation*}
\mu_0\into{S_R}|\nabla\vi_2|^2\le 2\into{T_{\bar R}}\partial_j{v_2}_i\Cc_{ijhk} u_h\partial_k\eta_{\bar R}- 2\into{T_{\bar R}}{v_2}_i\Cc_{ijhk}\partial_ku_h\partial_j\eta_{\bar R}={\mathcal J}_1+{\mathcal J}_2.
\end{equation*}
By Schwarz's  inequality
$$
|{\mathcal J}_1|^2\le c\into{S_R}|\nabla\vi_2|^2 \into{T_{\bar R}}|\u|^2\le  c\into{S_R}|\nabla\vi_2|^2,  
$$
and since by  (\ref{NRT0}) $\int_{T_{\bar R}}\Cc_{ijhk}\partial_ku_h\partial_j\eta_{\bar R}=0$,
$$
|{\mathcal J}_2|^2\le c \into{T_{\bar R}}\Big|\vi_2-{1\over |T_{\bar R}|}\int_{T_{\bar R}}\vi_2\Big|^2 \into{T_{\bar R}}|\Ci[\nabla\u]|^2\le c\into{S_R}|\nabla\vi_2|^2.
$$
Hence
\begin{equation}
\label{Saa3}
\into{  S_R}|\nabla\vi_2|^2\le c_0.
\end{equation}
By uniqueness $\vi=\vi_1+\vi_2$ in $S_R$. Therefore, putting together (\ref{Saa2}), (\ref{Saa3}), using the inequality $|a+b|^2\le 2|a|^2+2|b|^2$ and Lemma \ref{Cacciol}, we get
\begin{equation}
\label{sstew23}
\begin{array}{r@{}l}
\displaystyle\into {S_\rho }|\nabla \vi|^2 & {} \displaystyle\le 2\into {S_\rho }(|\nabla \vi_1|^2
+|\nabla\vi_2|^2) 
\le c\left({\rho\over R}\right)^\gamma\into {S_R}|\nabla \vi_1|^2+c_0  \\[15pt]
 & {}   \displaystyle  \le c\left({\rho\over R}\right)^\gamma\into {S_R}|\nabla \vi |^2+c_0\le {c(\rho)\over R^{2+\gamma}}\into{T_R}|\u|^2+ c.
\end{array}
\end{equation}
Hence, taking into account (\ref{MBazzzr}), letting $R\rightarrow+\infty$, we obtain $\nabla\u\in L^2(\Omega)$.

Let consider now the function (\ref{function1}). Multiplying (\ref{conell})$_1$ scalarly by $g_R\u$ and integrating by parts, we get
\begin{equation}
\label{uguag}
\into{\Omega}g_R\nabla\u\cdot\C[\nabla\u]=\into{\partial\Omega}\u\cdot\s(\u)-\into {T_R}\nabla g_R\cdot\C[\nabla\u]\u.
\end{equation}
From (\ref{NRT0}) it follows that $\into{T_R}\C[\nabla\u]\e_r={\bf 0}$, so that by applying Schwarz's inequality and Poincar\'e's inequality 
\begin{equation*}
\Big|\displaystyle\into{\Omega}\nabla g_R\cdot\C[\nabla\u]\u\Big|\le \displaystyle{c\over R}\Big(\into{T_R}|\u-\u_{T_R}|^2\Big)^{1\over 2}\Big(\into{T_R}|\nabla \u|^2\Big)^{1\over 2} \le c \into{T_R}|\nabla \u|^2.
\end{equation*}
Therefore, (\ref{Nxx0}) follows from (\ref{uguag}) by letting $R\to +\infty$ and taking into account the properties of $g_R$ and that $\nabla\u\in L^2(\Omega)$.

\end{proof}

\begin{remark} 
\label{Dqsol}
In the previous Lemma we proved, in particular,  that a variational solution which satisfies \eqref{NRT0} and \eqref{MBazzzr} is a $D$--solution. Another sufficient condition to have a $D$-solution is to assume \eqref{NRT0} and $\u\in D^{1,q}(\complement S_{R_0})$, for some $R_0$ sufficiently large and for some $q\in \left ( 2, \frac 4 {2-\gamma}\right)$.  Indeed, by reasoning as in  \eqref{sstew23} and applying H\"older's inequality we obtain
$$
\into {S_\rho }|\nabla \vi|^2   \le c\left({\rho\over R}\right)^\gamma\into {S_R}|\nabla \vi |^2+c_0\le c(\rho) R^{\frac {2(q-2)}{q}-\gamma}\left (\int_{\complement S_{R_0}}|\nabla \vi|^q\right )^{2/q}+c_0\,.
$$
Then we get $\nabla\u\in L^2(\Omega)$ on letting $R\to+\infty$.
\end{remark}

\begin{remark} 
From Lemma \ref{Lemmm34} it follows that up to a constant  the homogeneous traction problem
\begin{equation*}
\begin{array}{r@{}l}
\div{\C}[\nabla \u]& {} ={\0}\!\!\qquad\qquad \hbox{\rm in }\Omega, \\[4pt]
\s(\u)&{} ={\0}\!\!\qquad\qquad \hbox{\rm on }\partial\Omega,\\[4pt]
\u(x)& {} =o(r^{\gamma/2}),\end{array}  
\end{equation*}
has only the trivial solution.
\end{remark}

\begin{lemma}  
\label{Lem2}
A D--solution $\u$ to   $(\ref{conell})_1$ satisfies $(\ref{NRT0} )$ and for all $R>\rho\gg R_0$,
\begin{equation}
\label{Curti}
\into{\complement  S_R}|\nabla\u|^2\le c\left({\rho\over R}\right)^{\gamma}\into{\complement  S_\rho}|\nabla\u|^2.
\end{equation}
\end{lemma}
\begin{proof}
As in the proof of Lemma \ref{Lem1}, it is sufficient to assume $\u$ regular. Multiplying $(\ref{conell})_1$ by the function (\ref{function1})  and integrating over  $\Omega$, we have
$$
\into {\partial\Omega}\s({\u})=\into {T_R}{\C}[\nabla{\u}]\nabla g_R.
$$
Hence  (\ref{NRT0}) follows   taking into account that by Schwarz's inequality
$$
 \left|\,\into {\partial\Omega}\s({\u})\right| =\left|\,\into {T_R}{\C}[\nabla{\u}]\nabla g_R\right|
 \le    {1\over R}\left\{\,\into {T_R}
|\nabla{\u}|^2\right\}^{1/2}\left\{\,\into {T_R}\right\}^{1/2}\le c\|\nabla\u\|_{L^2(T_R)},$$
and letting $R\to+\infty$.  

A standard computation yields
$$
\begin{array}{l}
\displaystyle \mu_0\into{\complement S_R}  g_{\varrho}\big(|\nabla\u|^2+|\div\u|^2\big)  \displaystyle=2\mu_0\into{\complement S_R}  g_{\varrho}|\hat\nabla\u|^2+\mu_0\into{\partial S_R}\e_R\cdot[\nabla\u  -(\div\u){\bf 1}] \u \\[20pt]
\qquad\displaystyle -{{\mu_0}\over {\varrho}}\into{T_{\varrho}}\e_r \cdot[\nabla\u  -(\div\u){\bf 1}] \u\le  - \into{\partial S_R}\e_R\cdot[2\Ci[\nabla\u]-\mu_0(\nabla\u  -(\div\u){\bf 1})] \u \\[20pt]
\qquad\displaystyle +{1\over {\varrho}}\into{T_{\varrho}}\e_r\cdot[2\Ci[\nabla\u]-\mu_0(\nabla\u  -(\div\u){\bf 1})] \u, 
\end{array}
$$
for $\varrho\gg R_0$. Hence, since by (\ref{Serv}), Schwarz's inequality and Wirtinger's inequality
$$
\left|{1\over {\varrho}}\into{T_{\varrho}}\e_r\cdot[2\Ci[\nabla\u]-\mu_0(\nabla\u  -(\div\u){\bf 1})] \u\right|\le c\|\nabla\u\|_{L^2(T_{\varrho})},
$$
letting $\varrho\to+\infty$, it follows
\begin{equation}
\label{ASERTA}
 \mu_0\into{\complement S_R}   \big(|\nabla\u|^2+|\div\u|^2\big)  \displaystyle \le  - \into{\partial S_R}\e_R\cdot[2\Ci[\nabla\u]-\mu_0(\nabla\u  -(\div\u){\bf 1})] \u.  \end{equation}
Now  proceeding as we did in the proof of Lemma \ref{Lem1},    (\ref{ASERTA}) yields
\begin{equation}
\label{Casagiove}
\gamma Q(R)=\gamma \into{\complement S_R}   \big(|\nabla\u|^2+|\div\u|^2\big) \le R\into{\partial S_R}   \big(|\nabla\u|^2+|\div\u|^2\big).  \end{equation}
 Since by the basic calculus
 $$
Q'(R)=-\into{\partial S_R}\big(|\nabla\u|^2+|\div\u|^2\big),
$$
(\ref{Curti}) follows from (\ref{Casagiove}) by a simple integration.
\end{proof}
 
 \begin{lemma}
\label{Lem3}
There is $\epsilon=\epsilon(\gamma)>0$ such that every $D$--solution $\u$ to $(\ref{conell})_{1}$  belongs to $D^{1,q}({\mathcal I})$ for all $q\in (2-\epsilon,2+\epsilon)$. Moreover, if $\Ci$ is regular at infinity, then $\u$  $D^{1,q}({\mathcal I})$ for all $q\in(1,+\infty)$. 
\end{lemma}
\begin{proof}
Let $\eta_{\bar R}$ be the function   (\ref{function2}) for large $\bar R$. The field $\vi= \eta_{\bar R}\u$   is a variational solution to 
\begin{equation*}
\div\Ci_0[\nabla\vi]+\div(\Ci-\Ci_0)[\nabla\vi]+\f={\bf 0}\quad \hbox{\rm in }{\Bbb R}^2,
\end{equation*}
where $\f$ is defined by (\ref{Bxvvfrw}) and $\Ci_0$ is a  constant elasticity tensor. Let $\boldsymbol{\mathscr{U}}(x-y)$ be the fundamental solution to the equation $\div\Ci_0[\nabla\vi]=\0$,
the integral transform
$$
\boldsymbol{\mathscr{Q}}[\vi](x)=\nabla\into{{\Bbb R}^2} \boldsymbol{\mathscr{U}}(x-y) ( \Ci-\Ci_0)[\nabla\vi](y)dv_y
$$
maps  $D^{1,q}$ into itself for every $q\in(1,+\infty)$. Set
$$
\vi_f(x)=\into{{\Bbb R}^2} \boldsymbol{\mathscr{U}}(x-y)\f(y)dv_y\in D^{1,q}({\Bbb R}^2),\quad\forall  q\in (1,+\infty),
$$
and consider the functional equation 
\begin{equation}
\label{Dsq2sl}
\vi'(x)=\vi_f(x)+ \boldsymbol{\mathscr{Q}} [\vi](x).
\end{equation} 
Choose
$$
{\Cc}_0{_{ijhk}}= \mu_e\delta_{ih}\delta_{jk}. 
$$
Since \cite{CampanatoL}
$$
\big\|\boldsymbol{\mathscr{Q}}[\vi]\big\|_{  D^{1,q} }\le c(q){\mu_e-\mu_0\over \mu_e}\|\vi\|_{  D^{1,q} }
$$
and 
$$
\lim_{q\to 2}c(q)=1,
$$
there is $\epsilon>0$ such that  (\ref{Dsq2sl})  is a contraction in $  D^{1, q}$, $q\in (2-\epsilon,2+\epsilon)$.   If $\Ci$ is regular at infinity, then, choosing $\bar R$ large as we want, we can do $|\Ci(x)-\Ci_0|$ arbitrarily small and, as a consequence, $\|\boldsymbol{\mathscr{Q}}[\vi]\|_{D^{1,q}}\le \beta\|\vi\|_{D^{1,q}}$, for  every  positive $\beta$ and this is sufficient to conclude the proof. 
\end{proof}

\medskip

Extend $\C$ to the whole of ${\Bbb R}^2$ by setting $\C=\bar {\C} $ in $\complement\Omega$ (say), with $\bar {\C}$ constant and positive definite. Clearly, the new elasticity tensor (we denote by the same symbol)  satisfies (\ref{Sxxa}) (almost everywhere) in ${\Bbb R}^2$. 

The H\"older regularity of variational solutions to (\ref{conell})$_1$ is sufficient to prove the unique existence of a fundamental (or Green) function $\Q(x,y)$ to (\ref{conell})$_1$ in ${\Bbb R}^2$ (see  \cite{Chanillo}, \cite{ Dong}, \cite{Kenig}, \cite{TKB}), which satisfies
$$
{\bmit\varphi}(x)=\into {{\Bbb R}^2}\nabla{\bmit\varphi}(y)\cdot\C[\nabla{\bmit G}(x,y)]da_y,
$$
for all ${\bmit\varphi}\in C^\infty_0({\Bbb R}^2)$.  It is a variational solution to (\ref{conell})$_1$  in $x$ [resp. in $y$] in every domain not containing $y$ [resp. $x$]. Moreover,  $\Q(x,y)=\Q^\top(y,x)$ and for $\f\in {\mathcal H}^1({\Bbb R}^2)$ the field
 $$
 \u(x)=\into {{\Bbb R}^2}{\bmit G}(x,y){\bmit f}(y)da_y\in D^{1,2}({\Bbb R}^2)\cap C^{0,\mu}_{\rm loc}({\Bbb R}^2)
 $$
 is the unique variational solution to 
\begin{equation}
\label{kkki89}
 \div\C[\nabla\u]+\f={\bf 0}\quad\hbox{\rm in }{\Bbb R}^2.
\end{equation}
 
 ${\bmit G}(x,\cdot)$ belongs to the John--Niremberg space $BMO({\Bbb R}^2)$  (see, $e.g$.,  \cite{Giusti}) and has a logarithm  singularity at $x$ and at infinity. Set $\w(x)={\bmit G}(x,o)\e$, with $\e$ constant vector. Let us show that 
$\nabla\w\not\in L^2(\complement S_{R_0})$ and $\nabla\w \in L^q(\complement S_{R_0})$ for all $q$ in a right neighborhood of 2. Indeed, if $\w\in D^{1,2}(\complement S_{R_0})$, then, by applying \eqref{CaccioppoliOut} and H\"older's inequality, we get
$$
\into {\complement S_R}|\nabla\w|^2\le  CR^{-4/q}\left\{\into {T_R}|\w|^q\right\}^{2/q},\quad q>2.
$$
Therefore, from (\ref{Curti}) it follows
\begin{equation*}
\into {\complement S_{R}} |\nabla\w|^2  \le  {c\rho^{\gamma-4/q}\over R^\gamma}\into {T_\rho}|\w|^q.
 \end{equation*} 
Hence, choosing $q>4/\gamma$, letting $\rho\to 0$   and taking into account that $\w\in L^q_{\rm loc}({\Bbb R}^2)$, we have the contradiction $  \nabla\w={\bf 0}$. 
The field $\vi= \eta_{R_0}\w$ is a solution to (\ref{kkki89})
   where $\eta_R$  and $\f$ are  defined by (\ref{function2}), (\ref{Bxvvfrw}), respectively. 
 By well--known estimates \cite{TKB} and \eqref{Caccioppoli} for large $R$, we have 
  $$
 \begin{array}{r@{}l}
 \displaystyle\left(\into {S_R}|\nabla\vi|^q\right)^{1\over q} & {} \displaystyle\le c\left\{R^{-1+2/q} \left(\into {S_R}|\nabla\vi|^2\right)^{1\over 2}+c_{\f}\right\}\\[12pt]
 & {} \displaystyle\le c\left\{R^{-2+2/q} \left(\into {T_R}|\w|^2\right)^{1\over 2}+c_{\f}\right\},
 \end{array}
 $$
for $q\in (2,\bar q)$, with $\bar q>2$ depending on $\mu_0$, where $c_{\f}$ is a constant depending on $\f$.

Hence, letting $R\to+\infty$ and bearing in mind the behavior of $\w$ at large distance, it follows that  $\nabla\w\in L^q(\complement S_{R_0})$. 
 Collecting the above results we can say that the fundamental function satisfies:

\begin{itemize}
\item[(\i)] ${\bmit G}(x,y)\not\in D^{1,2}(\complement S_{R}(x))$ for all $R>0$;
\item[(\i\i)] ${\bmit G}(x,y)\in D^{1,q}(\complement S_{R_0}(x))$,  for all $q\in (2,\bar q)$, with  $\bar q>2$ depending on $\mu_0$.
\end{itemize} 
 \section{Proof of Theorems \ref{T1}, \ref{T3} }
 \label{PT4}
{\sc Proof of Theorem \ref{T1}}. $(i)$  --  If $\h(\ne{\bf 0})\in \boldsymbol{\frak M}$, then  $\into {\partial\Omega}\s(\h)\ne \0$, otherwise, bearing in mind that $\h\in BMO$, Caccioppoli's inequality writes
$$
\into{S_{R/2}}|\nabla\h|^2\le {c\over R^2} \into{S_R}\left|\h-{1\over |S_R|}\int_{S_R}\h\right|^2\le c,
$$
for some $c$ independent of $R$. Hence $\h$ should be a $D$--solution and so by uniqueness $\h=\0$.
Let $\u_i\in D^{1,2}(\Omega)$ $(i=1,2)$ be the solutions to (\ref{conell})$_{1,2}$ with $\hat{\u}_i=-{\bmit G}(x,o)\e_i$ and set $\h_i=\u_i+{\bmit G}(x,o)\e_i$.  If 
$\alpha_i\h_i=\0$, then $\alpha_i{\bmit G}(x,o)\e_i\in D^{1,2}(\Omega)$ and this is possible if and only if $\alpha_i\e_i=\0$, $i.e.$ $\alpha_i=0$ and the system $\{\h_1,\h_2\}$ is linearly independent. Therefore $\hbox{\rm dim}\,\boldsymbol{\frak M}\ge 2$. Clearly, for every $\{\h_1,\h_2,\h_3\}\subset{\frak M}$ the system   $\{\into {\partial\Omega}\s(\h_i)\}_{i\in\{1,2,3\}}$ is linear dependent. Therefore, there are (not all  zero) scalars $\alpha_i$ such that $\into {\partial\Omega}\s(\h)=\0$, with $\h=\alpha_i\h_i$. Since this implies that $\h=\0$,  we conclude that $\hbox{\rm dim}\,\boldsymbol{\frak M}=2$. It is obvious that if  $\{\h_1,\h_2\}$ is a basis of $\boldsymbol{\frak M}$, then $\Big\{\into {\partial\Omega}\s(\h_1), \into {\partial\Omega}\s(\h_2)\Big\}$ is a basis of ${\Bbb R}^2$. 

\smallskip

  $(ii)$  -- Multiply (\ref{conell})$_1$ scalarly by $g_R\h$, with $\h\in{\frak M}$. Integrating
by parts  we get 
\begin{equation}
\label{kkkaft5}
\begin{array}{r@{}l}
\displaystyle\into {\partial\Omega}(\hat{\u}-\u_0)\cdot \s(\h) & {} \displaystyle=-{1\over R}  \into {T_R}(\u-\u_0)\cdot{\C}[\nabla
\h]\e_R\\[15pt]
& {} \displaystyle+{1\over R}\into {T_R}\h \cdot{\C}[\nabla \u]\e_R.
\end{array}
\end{equation}

Choosing $s(<2)$ very close to $2$ we have
$$
\begin{array}{l}
\displaystyle {1\over R}\left|\ \into {T_R}(\u-\u_0)\cdot{\C}[\nabla \h]\e_R\right|
  \le  c\left\{\into {T_R}{|\u-\u_0|^s\over
r^s}\right\}^{1/s}\left\{\into {\Omega} |\nabla \h|^{s'} \right\}^{1/s'}\\[12pt]
\displaystyle {1\over R}\left|\ \into {T_R}\h\cdot{\C}[\nabla \u]\e_R\right|
  \le  c\left\{\into {T_R}{|\h  |^{s'}\over
r^{s'}}\right\}^{1/s'}\left\{\into {\Omega} |\nabla \u|^{s } \right\}^{1/s }.
\end{array}
$$
Therefore, letting  $R\to+\infty$ in (\ref{kkkaft5}), in virtue of Lemma \ref{Inequal} and \ref{Lem3} and the properties of ${\bmit G}$, we see that  
\begin{equation*}
\into {\partial\Omega}(
\hat{\u}-\u_0)\cdot \s(\h)={\bf 0},\quad\forall\, \h\in{\frak M}.
\end{equation*}
Hence it follows that  $\u_0=\0$ if and only if $\hat{\u}$ satisfies (\ref{defu0}).

\medskip

 $(iii)$  -- If $\u=o(r^{\gamma/2})$ is a nonzero  variational solution to (\ref{conell})$_1$, vanishing on $\partial\Omega$, then there are scalars $\alpha_1$ and $\alpha_2$ such that
$$
\into{\partial\Omega}\s(\u)=\alpha_1\into{\partial\Omega}\s(\h_1)+\alpha_2\into{\partial\Omega}\s(\h_2),
$$
where $\{\h_1,\h_2\}$ is a basis of ${\frak M}$.
Therefore, by (\ref{Morr2}) and \eqref{Caccioppoli} the field $\vi=\u-\alpha_1\h_1-\alpha_2\h_2$ satisfies
$$
\into{  \Omega_\rho}|\nabla\vi|^2\le  {c\rho^\gamma\over R^{2+\gamma}}   \into{ T_R}| {\bmit v}|^2.
$$
Hence, letting $R\to+\infty$, it follows that $\u\in\boldsymbol{\frak M}$. Clearly, if $\u(x)=o(\log r)$, then $\u=\0$.

\medskip

 $(iv)$  --  Let $R<|\x|<2R$, $R\gg R_0$,   let $\mathscr{A}$ be a neighborhood of $x$. By H\"older's inequality and Sobolev's inequality
$$
 \into \mathscr{A}|\u|^2\le c \left\{\into {\mathscr{A}}|\u|^{2q/(2-q)}\right\}^{(2-q)/q}\le c
\left\{\into {\complement S_R}|\nabla\u|^q\right\}^{2/q},
$$
for $q\in (2-\epsilon(\gamma),2)$. Hence by the classical {\it convexity inequality\/}
$$
\|\nabla\u\|_{L^q(\complement S_R)}\le \|\nabla\u\|_{L^s(\complement S_R)}^\theta\|\nabla\u\|^{1-\theta}_{L^2(\complement S_R)}, 
$$
with $2-\epsilon(\gamma)<s<q$, $\;\theta= s(2-q)/q(2-s)$, taking into account Lemma \ref{Lem2} and \ref{Lem3}, it follows 
\begin{equation}
\label{ksss31}
 \into {\mathscr{A}}|\u|^2\le cR^{( \theta-1)\gamma}.
\end{equation} 
Putting together (\ref{Curti}), (\ref{ksss31}), we have
$$
\into {\mathscr{A}}|\u|^2+{1\over \rho^{2+\gamma}}\into {S_\rho(x)}|\u-\u_{S_\rho(x)}|^2\le cR^{( \theta-1)\gamma}.
$$
Hence (\ref{desss1}) follows taking into account well--known results of S. Campanato (see, $e.g.$, \cite{Giusti} Theorem 2.9) and that $\theta\to 0$ for $q\to2$.

\smallskip

Let now $\C$ satisfy  (\ref{lkj˜jkskj9}) and let $\u'$, $\u''$ be the variational solutions to the systems
\begin{equation*}
\begin{array}{r@{}l}
 \div{\C_0}[\nabla \u']  & {} ={ 0}\quad\hbox{\rm in
}S_R(x) ,\\[2pt]
\u'& {} = \u \quad\hbox{\rm on }\partial S_R(x) ,
\end{array}
\end{equation*}
and 
\begin{equation*}
\begin{array}{r@{}l}
 \div{\C_0} [\nabla \u'']+\div(\C-\C_0) [\nabla \u] & {} ={ \0}\quad\hbox{\rm in
}S_R(x) ,\\[2pt]
 \u'' & {} ={\bf 0} \quad\hbox{\rm on }\partial S_R(x),
\end{array}
\end{equation*}
respectively.
Applying Poincar\'e's and Caccioppoli's inequalities  we have 
$$
\into {S_R(x)}| \u''|^2\le  R^2\into {S_R(x)}|\nabla\u''|^2\le c(\epsilon)R^2\into {S_R(x)}|\nabla\u|^2\le c(\epsilon) \into {T_R(x)}| \u|^2.
$$
Hence, taking into account that
$$
\into {S_\rho(x)}|\u'|^{2 }\le c\left({\rho\over R}\right)^{2 }\into {S_R(x)}|\u'|^2,
$$
 it follows \cite{CampanatoL} (see also \cite{RUTA})
\begin{equation}
 \label{KKK2}
\into {S_\rho(x)}| \u|^2  \le c\left({\rho\over R}\right)^{2-\epsilon }\into {S_R(x)}| \u|^2 . 
\end{equation} 
Putting together (\ref{KKK2}) and H\"older's inequality 
$$
\into {S_R(x)}|\u|^2\le cR^{{2(s-2)}/s}\left\{\into {S_R(x)}|\u|^s\right\}^{2/s},
$$
for $s>2$, we get
\begin{equation}
 \label{KKK3}
\begin{array}{r }
\displaystyle\into {S_\rho(x)}|\u|^2 + {1\over \rho^{4-\epsilon}}\into {S_\rho(x)}|\u-\u_{S_\rho(x)}|^2 \le   {c\over R^{2-\epsilon}}\into {S_R(x)}|\nabla\u|^2\\[12pt]
 \displaystyle+cR^{\epsilon-2+2(s-2)/s}\left\{\into {S_R(x)}|\u|^s\right\}^{2/s}.
 \end{array}
\end{equation} 
Since we can choose $s(>2)$ near to 2 as we want, (\ref{KKK3}) yields
$$
|\u(x)|\le {c\over |\x|^{1-\epsilon}},
$$
for all positive $\epsilon$.

\hfill$\square$

\bigskip

{\sc Proof of Theorem \ref{T3}}. If $\C$ satisfies the stronger assumption (\ref{SEaabb}), by the argument in \cite{PS} one shows that
a variational solution to $\div\C[\nabla\u]=\0$ in $S_R(x)$ satisfies
\begin{equation*}
\into {S_\rho(x)}| \nabla\u|^2  \le c\left({\rho\over R}\right)^{2/\sqrt L }\into {S_R(x)}|\nabla\u|^2, 
\end{equation*} 
for every $\rho\in(0,R]$ and the Lemmas hold with $\gamma$ replaced by $2/{\sqrt L}$. Hence the desired results follow by repeating the steps in the proof of Theorem \ref{T1}.

\hfill$\square$
\section{A Counter--example }
\label{DeGC}
 The following slight modification of a famous counter--example by E. De Giorgi \cite{DeGiorgi}  assures that the uniqueness
class  in Theorem \ref{T3} and the rates of decay are sharp. 
 
  Let $\tilde{\C}$ be the symmetric elasticity tensor defined by
 \begin{equation*}
 \tilde{\C}[{ \L}]=  \hbox{\rm sym}\,\L +4\xi^{-2}(\e_r\otimes\e_r)(\e_r\cdot\L\e_r),\quad\xi\ne 0,\ {\bmit L}\in {\rm{Lin}}.
 \end{equation*}
Clearly,  $\tilde{\C}$ is bounded on ${\Bbb R}^2$ and
 $C^\infty$ on
 ${\Bbb R}^2\setminus\{o\}$. Since 
 $${ \L}\cdot\tilde{\C}[{ \L}]=4\xi^{-2}|\e_r\cdot\L\e_r|^2+|\L|^2,\quad\forall\,{\L}\in
 \hbox{\rm Sym},
 $$ 
 $\tilde{\C}$ satisfies \eqref{Sxxa} with $\mu_0=1$ and $\mu_e=1+4\xi^{-2}$.  A simple computation  \cite{DeGiorgi} shows
that the equation
 \begin{equation*}
 \div\tilde{\C}[\nabla{\u}]= { \0} 
 \end{equation*}
  admits the family of solutions 
 \begin{equation*}
\u'= ( c_1r^\epsilon+c_2r^{-\epsilon}) \e_r,
 \end{equation*}
with  
$$
\epsilon= 
{|\xi|\over\sqrt{4+\xi^2}},
$$  
for every $c_1$, $c_2\in {\Bbb R}$.  Of course, for $c_1=1$, $c_2=-1$, $\u'={ \0}$ on $\partial S_1$ and $\u'\in D^{1,q}(\complement S_1)$ for
$q>2/(1-\epsilon)$,
$ \u'\not\in D^{1,q}(\complement S_1)$ for $q\le2/(1-\epsilon)$ so that, in particular, bearing in mind the properties of ${\bmit G}$, $\u'\not\in{\frak M}$. 
   For differential systems satisfying the stronger assumption (\ref{SEaabb})
the above example shows that the decay  $\u-\u_0=o(r^{{1/\sqrt L}})\}$ is optimal for  $D$--solutions  and the class $\{\u:\u=o(r^{1/\sqrt L})\}$ 
is    borderline   for uniqueness of the variational solution to the Dirichlet problem up to a field of ${\frak M}$.

\vfill\eject

\end{document}